\documentclass[12pt]{amsart}

\newtheorem{thm}{Theorem}[section]

\newtheorem{cor}[thm]{Corollary}

\newtheorem{prop}[thm]{Proposition}
\newtheorem*{prob*}{Problem}

\newtheorem*{thm*}{Theorem}

\theoremstyle{definition}
\newtheorem{defn}[thm]{Definition}

\newtheorem*{defn*}{Definition}
\newtheorem{rem}[thm]{Remark}
\numberwithin{equation}{section}

\newcommand{\bx}{\breve{x}}
\newcommand{\by}{\breve{y}}

\newcommand{\C}{\mathbb C}

\newcommand{\Y}{\mathbb Y}
\newcommand{\Z}{\mathbb Z}
\newcommand{\Zp}{\mathbb Z_{\geq 0}}

\newcommand{\X}{\mathfrak{X}}

\DeclareMathOperator{\diag}{diag}

\begin{document}
\title[$z$-measures]
 {\bf{$Z$-measures on partitions related to the infinite Gelfand pair $\left(S(2\infty),
H(\infty)\right)$}}

\author{Eugene Strahov}

\begin{abstract}
The paper deals with the $z$-measures on partitions with the
deformation (Jack) parameters 2 or 1/2. We provide a detailed
explanation of the representation-theoretic origin of these
measures, and of their role in the harmonic analysis on the
infinite symmetric group.
$$
$$
\end{abstract}

\thanks{
Department of Mathematics, The Hebrew University of Jerusalem,
Givat Ram, Jerusalem 91904. E-mail: strahov@math.huji.ac.il.
Supported by US-Israel Binational Science Foundation (BSF) Grant
No. 2006333,
 and by Israel Science Foundation (ISF) Grant No. 0397937.\\
}
\maketitle \section{Introduction} Let $S(\infty)$ denote the group
whose elements are finite permutations of $\{1,2,3,\ldots\}$. The
group $S(\infty)$ is called the infinite symmetric group, and it
is a model example of a "big" group. The harmonic analysis for
such groups  is an active topic of modern research, with
connections to different areas of mathematics from enumerative
combinatorics to random growth models and to the theory of
Painlev\'e equations.  A theory of harmonic analysis on the
infinite symmetric  and infinite-dimensional unitary groups is
developed  by Kerov, Olshanski and Vershik \cite{KOV1, KOV2},
Borodin \cite{Borodin-1}, Borodin and Olshanski
\cite{Borodin-2,Borodin-3}. For an introduction to harmonic
analysis on the infinite symmetric group see Olshanski
\cite{olshanski2003}. Paper by Borodin and Deift
\cite{BorodinDeift} studies differential equations arising in the
context of harmonic analysis on the infinite-dimensional unitary
group, and paper by Borodin and Olshanski \cite{BorodinCombin}
describes the link to problems of enumerative combinatorics, and
to certain random growth models. For very recent works on the
subject see, for example, Vershik and  Tsilevich \cite{vershik},
Borodin and Kuan \cite{BorodinKuan}.

Set
$$ G=S(\infty)\times S(\infty),
$$
$$
K=\diag S(\infty)=\left\{(g,g)\in G \mid g\in
S(\infty)\right\}\subset G.
$$
Then $(G,K)$ is an infinite dimensional Gelfand pair in the sense
of Olshanski \cite{olshanskiGelfandPairs}. It can be shown that
the biregular spherical representation of $(G,K)$ in the space
$\ell^2\left(S(\infty)\right)$ is irreducible. Thus the
conventional scheme of noncommutative harmonic analysis is not
applicable to the case of the infinite symmetric group.

In 1993, Kerov, Olshanski and Vershik \cite{KOV1} (Kerov,
Olshanski and Vershik \cite{KOV2} contains the details)
 constructed a family $\{T_z: z\in\C\}$ of unitary representations
 of the bisymmetric infinite group $G=S(\infty)\times
 S(\infty)$. Each representation $T_z$ acts in the Hilbert space
 $L^2(\mathfrak{S},\mu_t)$, where $\mathfrak{S}$ is a certain compact space called the space of
 virtual permutations, and $\mu_t$ is a distinguished $G$-invariant probability measure on
 $\mathfrak{S}$ (here $t=|z|^2$). The representations $T_z$
 (called the generalized regular representations) are reducible.
 Moreover, it is possible to extend the definition of $T_z$ to the
 limit values $z=0$ and $z=\infty$, and it turns out that
 $T_{\infty}$ is equivalent to the biregular representation
 of $S(\infty)\times S(\infty)$. Thus, the family $\{T_z\}$ can be
 viewed as a deformation of the biregular representation.
 Once the representations $T_z$
 are constructed, the main problem of the harmonic analysis on the
 infinite symmetric group is in decomposition of the generalized
 regular representations $T_z$ into irreducible ones.

 One of the initial steps in this direction can be described as
 follows. Let $\textbf{1}$ denote the function on $\mathfrak{S}$
 identically equal to 1. Consider this function as a vector of
 $L^2(\mathfrak{S},\mu_t)$. Then $\textbf{1}$ is a spherical
 vector, and the pair $(T_z,\textbf{1})$ is a spherical
 representation of the pair $(G,K)$, see, for example, Olshanski
 \cite{olshanski2003}, Section 2. The spherical function of
 $(T_z,\textbf{1})$ is the matrix coefficient
 $(T_z(g_1,g_2)\textbf{1},\textbf{1})$, where $(g_1,g_2)\in
 S(\infty)\times S(\infty)$. Set
 $$
 \chi_z(g)=\left(T_z(g,e)\textbf{1},\textbf{1}\right),\; g\in
 S(\infty).
 $$
 The function $\chi_z$ can be understood as a character of the
 group $S(\infty)$ corresponding to $T_z$. Kerov, Olshanski and
 Vershik \cite{KOV1,KOV2} found  the restriction
 of
 $\chi_z$ to $S(n)$ in terms of irreducible characters of $S(n)$.
 Namely, let $\Y_n$ be the set of Young diagrams with $n$ boxes.
 For $\lambda\in\Y_n$ denote by $\chi^{\lambda}$ the corresponding
 irreducible character of the symmetric group $S(n)$ of degree
 $n$. Then for any $n=1,2,\ldots$ the following formula holds true
 \begin{equation}\label{EquationHiDecomposition}
 \chi_z\biggl|_{S(n)}=\sum\limits_{\lambda\in\Y_n}M^{(n)}_{z,\bar{z}}(\lambda)\frac{\chi^{\lambda}}{\chi^{\lambda}(e)}.
 \end{equation}
In this formula $M^{(n)}_{z,\bar{z}}$ is a probability measure
(called the $z$-measure) on the set of Young diagrams with $n$
boxes, or on the set of integer partitions of $n$. Formula
(\ref{EquationHiDecomposition}) defines the $z$-measure
$M^{(n)}_{z,\bar{z}}$ as a weight attached to the corresponding
Young diagram in the decomposition of the restriction of $\chi_z$
to $S(n)$ in irreducible characters of $S(n)$. Expression
(\ref{EquationHiDecomposition}) enables to reduce the problem of
decomposition of $T_z$ into irreducible components to the problem
on the computation of spectral counterparts of
$M^{(n)}_{z,\bar{z}}$.

In addition to their role in the harmonic analysis on the infinite
symmetric group the $z$-measures described above are quite
interesting objects by themselves. It is possible to introduce
 more general objects, namely measures $M_{z,z'}^{(n)}$ on Young diagrams
with $n$ boxes. Such measures depend on two complex parameters
$z,z'$. If $z'=\bar{z}$, then $M_{z,z'}^{(n)}$ coincide with the
$z$-measures in equation (\ref{EquationHiDecomposition}). Under
suitable restrictions on $z$ and $z'$ the weights $M_{z,z'}^{(n)}$
are nonnegative and their sum is equal to 1. Thus $M^{(n)}_{z,z'}$
can be understood as probability measures on $\Y_n$.  For special
values of parameters $z,z'$ the $z$-measures turn into discrete
orthogonal polynomial ensembles which in turn related to
interesting probabilistic models, see Borodin and Olshanski
\cite{BorodinCombin}. In addition, the $z$-measures are a
particular case of the Schur measures introduced by Okounkov in
\cite{okounkov1}. The $z$-measures $M_{z,z'}^{(n)}$ were studied
in details in the series of papers by Borodin and Olshanski
\cite{Borodin-4, BorodinCombin, BO1, BO}, in Okounkov
\cite{okounkov0}, and in Borodin, Olshanski, and Strahov
\cite{BorodinOlshanskiStrahov}.

Moreover, as it follows from   Kerov \cite{kerov}, Borodin and
Olshanski \cite{BO1} it is natural to consider a deformation
$M_{z,z',\theta}^{(n)}$ of $M_{z,z'}^{(n)}$, where $\theta>0$ is
called the parameter of deformation (or the Jack parameter). Then
the measures $M^{(n)}_{z,z'}$ can be thought as the $z$-measures
with the Jack parameter $\theta=1$. It is shown in Borodin and
Olshanski \cite{BO1}, that $M_{z,z'}^{(n)}$ are in many ways
similar to log-gas (random-matrix) models with arbitrary
$\beta=2\theta$. In particular, if $\theta=2$ or $\theta=1/2$ one
expects that $M_{z,z'}^{(n)}$ will lead to Pfaffian point
processes, similar to ensembles of Random Matrix Theory of
$\beta=4$ or $\beta=1$ symmetry types, see Borodin and Strahov
\cite{BS}, Strahov \cite{strahov} for the available results in
this direction.

It is the  purpose of the present paper to describe the origin of
 $z$-measures with the Jack parameters $\theta=2$ and
$\theta=1/2$ in  the representation theory.  First we recall the
notion of the $z$-measures with an arbitrary Jack parameter
$\theta>0$.  Then we consider the symmetric group $S(2n)$ viewed
as the group of permutations of the set
$\{-n,\ldots,-1,1,\ldots,n\}$, and its subgroup $H(n)$ defined as
the centralizer of the product of transpositions $(-n,n),
(-n+1,n-1),\ldots ,(-1,1)$. The group $H(n)$ is called the
hyperoctahedral group of degree $n$. One knows that $(S(2n),
H(n))$ are Gelfand pairs, and their inductive limit,
$(S(2\infty),H(\infty))$, is an infinite dimensional Gelfand pair,
see Olshanski \cite{olshanskiGelfandPairs}.  We describe the
construction of a family of unitary spherical representations
$T_{z,\frac{1}{2}}$ of the infinite dimensional Gelfand pair
$(S(2\infty),H(\infty))$ and show that $z$-measures with the Jack
parameters $\theta=1/2$ appear as coefficients in the
decomposition of the spherical functions of $T_{z,\frac{1}{2}}$
into spherical functions of the Gelfand pair $(S(2n),H(n))$. Due
to the fact that $z$-measures with the Jack parameters $\theta=2$
and $\theta=1/2$ are related to each other in a very simple way,
see Proposition \ref{PropositionMSymmetries}, the construction
described above provides a representation-theoretic interpretation
for $z$-measures with the Jack parameter $\theta=2$ as well.
Therefore, it is natural to refer to such $z$-measures as to the
$z$-measures for the infinite dimensional Gelfand pair
$(S(2\infty),H(\infty))$, or, more precisely, as to the
$z$-measures of the representation $T_{z,\frac{1}{2}}$.

The fact that these measures play a role in the harmonic analysis
was mentioned in Borodin and Olshanski \cite{BO1}, and in our
explanation of this representation-theoretic aspect we used many
ideas from  Olshanski \cite{olshanskiletter}.

\textbf{Acknowledgements} I am   grateful to Grigori Olshanski for
numerous discussions and many valuable comments at different
stages of this work.

\section{The $z$-measures on partitions with the general
parameter $\theta>0$}\label{Sectionztheta} We use Macdonald
\cite{macdonald} as a basic reference for the notations related to
integer partitions and to symmetric functions. In particular,
every decomposition
$$
\lambda=(\lambda_1,\lambda_2,\ldots,\lambda_l):\;
n=\lambda_1+\lambda_2+\ldots+\lambda_{l},
$$
where $\lambda_1\geq\lambda_2\geq\ldots\geq\lambda_l$ are positive
integers, is called an integer partition. We identify integer
partitions with the corresponding Young diagrams.  The set of
Young diagrams with $n$ boxes  is denoted by $\Y_n$.

Following Borodin and Olshanski \cite{BO1}, Section 1, and Kerov
\cite{kerov} let $M_{z,z',\theta}^{(n)}$ be a complex measure on
$\Y_n$ defined by
\begin{equation}\label{EquationVer4zmeasuren}
M_{z,z',\theta}^{(n)}(\lambda)=\frac{n!(z)_{\lambda,\theta}(z')_{\lambda,\theta}}{(t)_nH(\lambda,\theta)H'(\lambda,\theta)},
\end{equation}
where $n=1,2,\ldots $, and where we use the following notation
\begin{itemize}
    \item $z,z'\in\C$ and $\theta>0$ are parameters, the parameter
    $t$ is defined by
    $$
    t=\frac{zz'}{\theta}.
    $$
    \item $(t)_n$ stands for the Pochhammer symbol,
    $$
    (t)_n=t(t+1)\ldots (t+n-1)=\frac{\Gamma(t+n)}{\Gamma(t)}.
    $$
    \item
    $(z)_{\lambda,\theta}$ is a multidemensional analogue of the
    Pochhammer symbol defined by
    $$
    (z)_{\lambda,\theta}=\prod\limits_{(i,j)\in\lambda}(z+(j-1)-(i-1)\theta)
    =\prod\limits_{i=1}^{l(\lambda)}(z-(i-1)\theta)_{\lambda_i}.
    $$
     Here $(i,j)\in\lambda$ stands for the box in the $i$th row
     and the $j$th column of the Young diagram $\lambda$, and we
     denote by $l(\lambda)$ the number of nonempty rows in the
     Young diagram $\lambda$.
    \item
    $$
    H(\lambda,\theta)=\prod\limits_{(i,j)\in\lambda}\left((\lambda_i-j)+(\lambda_j'-i)\theta+1\right),
   $$
   $$
     H'(\lambda,\theta)=\prod\limits_{(i,j)\in\lambda}\left((\lambda_i-j)+(\lambda_j'-i)\theta+\theta\right),
   $$
      where $\lambda'$ denotes the transposed diagram.
\end{itemize}
\begin{prop}\label{PropositionHH}
The following symmetry relations hold true
$$
H(\lambda,\theta)=\theta^{|\lambda|}H'(\lambda',\frac{1}{\theta}),\;\;(z)_{\lambda,\theta}
=(-\theta)^{|\lambda|}\left(-\frac{z}{\theta}\right)_{\lambda',\frac{1}{\theta}}.
$$
Here $|\lambda|$ stands for the number of boxes in the diagram
$\lambda$.
\end{prop}
\begin{proof}
These relations follow immediately from definitions of
$H(\lambda,\theta)$ and $(z)_{\lambda,\theta}$.
\end{proof}
\begin{prop}\label{PropositionMSymmetries}
We have
$$
M_{z,z',\theta}^{(n)}(\lambda)=M_{-z/\theta,-z'/\theta,1/\theta}^{(n)}(\lambda').
$$
\end{prop}
\begin{proof}
Use definition of $M_{z,z',\theta}^{(n)}(\lambda)$, equation
(\ref{EquationVer4zmeasuren}), and apply Proposition
\ref{PropositionHH}.
\end{proof}
\begin{prop}\label{Prop1.3}
We have
$$
\sum\limits_{\lambda\in\Y_n}M_{z,z',\theta}^{(n)}(\lambda)=1.
$$
\end{prop}
\begin{proof}
See Kerov \cite{kerov}, Borodin and Olshanski
\cite{BO1,BOHARMONICFUNCTIONS}.
\end{proof}
\begin{prop}
If parameters $z, z'$ satisfy one of the three conditions listed
below, then the measure $M_{z,z',\theta}^{(n)}$ defined by
expression (\ref{EquationVer4zmeasuren}) is a probability measure
on $Y_n$. The conditions are as follows.\begin{itemize}
    \item Principal series: either
$z\in\C\setminus(\Z_{\leq 0}+\Zp\theta)$ and $z'=\bar z$.
    \item The complementary series: the parameter $\theta$ is a rational number, and both $z,z'$
are real numbers lying in one of the intervals between two
consecutive numbers from the lattice $\Z+\Z\theta$.
    \item The degenerate series: $z,z'$ satisfy one of the
    following conditions\\
    (1) $(z=m\theta, z'>(m-1)\theta)$ or $(z'=m\theta,
    z>(m-1)\theta)$;\\
    (2) $(z=-m, z'<-m+1)$ or $(z'=-m,
    z<m-1)$.
\end{itemize}

\end{prop}
\begin{proof} See Propositions 1.2, 1.3 in Borodin and Olshanski
\cite{BO1}.
\end{proof}
Thus, if the  conditions in the Proposition above  are satisfied,
then $M_{z,z',\theta}^{(n)}$ is a probability measure defined on
$\Y_n$, as  follows from Proposition \ref{Prop1.3}. In case when
$z, z'$ are taken either from the principal series or the
complementary series we refer to $z, z'$ as to admissible
parameters of  the $z$-measure under considerations. We will refer
to $M_{z,z',\theta}^{(n)}(\lambda)$ as to the $z$-measure with the
deformation (Jack) parameter $\theta$.
\begin{rem}
When both $z,z'$ go to infinity, expression
(\ref{EquationVer4zmeasuren}) has a limit
\begin{equation}\label{EquationPlancherelInfy}
M_{\infty,\infty,\theta}^{(n)}(\lambda)=\frac{n!\theta^{n}}{H(\lambda,\theta)H'(\lambda,\theta)}
\end{equation}
called the Plancherel measure on $\Y_n$ with general $\theta>0$.
Statistics of the Plancherel measure with the general Jack
parameter $\theta>0$ is discussed in  many papers, see, for
example, a very recent paper by Matsumoto \cite{matsumoto}, and
references therein. Matsumoto \cite{matsumoto} compares limiting
distributions of rows of random partitions with distributions of
certain random variables from a traceless Gaussian
$\beta$-ensemble.
\end{rem}

\section{The spaces $ X(n)$ and their projective limit}
\subsection{The homogeneous space $X(n)=H(n)\setminus S(2n)$}
Let $S(2n)$ be the permutation group of $2n$ symbols realized as
that of the set  $\{-n,\ldots,-1,1,\ldots,n\}$. Let $\breve{t}\in
S(2n)$ be the product of the transpositions $(-n,n),(-n+1,n-1),
\ldots, (-1,1)$. By definition, the group $H(n)$ is the
centralizer of $\breve{t}$ in $S(2n)$. We can write
$$
H(n)=\left\{\sigma\biggl|\sigma\in S(2n), \sigma
\breve{t}\sigma^{-1}=\breve{t}\right\}.
$$
The group $H(n)$ is called the hyperoctahedral group of degree
$n$.

Set $X(n)=H(n)\setminus S(2n)$. Thus $X(n)$ is the space of right
cosets of the subgroup $H(n)$ in $S(2n)$.

It is not hard to check that the set $X(n)$ can be realized as the
set of all pairings  of $\{-n,\ldots,-1,1,\ldots,n\}$ into $n$
unordered pairs. Thus every element $\breve{x}$ of $X(n)$ is
representable as a collection of $n$ unordered pairs,
\begin{equation}\label{representationofelement}
\breve{x}\in X(n)\longleftrightarrow
\breve{x}=\biggl\{\{i_1,i_2\},\ldots,\{i_{2n-1},i_{2n}\}\biggr\},
\end{equation}
where $i_1,i_2,\ldots,i_{2n}$  are distinct elements of the set
$\{-n,\ldots,-1,1,\ldots,n\}$.

For example, if $n=2$, then $S(4)$ is the permutation group of
$\{-2,-1,1,2\}$, the element $\breve{t}$ is the product of
transpositions $(-2,-1)$ and $(1,2)$, the subgroup $H(2)$ is
\begin{equation}
\begin{split}
H(2)=\biggl\{&\left(
\begin{array}{cccc}
  -2 & -1 & 1 & 2 \\
  -2 & -1 & 1 & 2 \\
\end{array}\right),\; \left(
\begin{array}{cccc}
  -2 & -1 & 1 & 2 \\
  2 & -1 & 1 & -2 \\
\end{array}\right),\\
&\left(
\begin{array}{cccc}
  -2 & -1 & 1 & 2 \\
  -2 & 1 & -1 & 2 \\
\end{array}\right),\;
\left(
\begin{array}{cccc}
  -2 & -1 & 1 & 2 \\
  2 & 1 & -1 & -2 \\
\end{array}\right),\\
&\left(
\begin{array}{cccc}
  -2 & -1 & 1 & 2 \\
  -1 & -2 & 2 & 1 \\
\end{array}\right),\;
\left(
\begin{array}{cccc}
  -2 & -1 & 1 & 2 \\
  1 & -2 & 2 & -1 \\
\end{array}\right),\\
&\left(
\begin{array}{cccc}
  -2 & -1 & 1 & 2 \\
  -1 & 2 & -2 & 1 \\
\end{array}\right),\;
\left(
\begin{array}{cccc}
  -2 & -1 & 1 & 2 \\
  1 & 2 & -2 & -1 \\
\end{array}\right)
\biggr\},
\end{split}
\nonumber
\end{equation}
and the set $X(2)$ is the set consisting of three elements, namely
$$
\biggl\{\{-2,-1\},\{1,2\}\biggr\},
\biggl\{\{-2,1\},\{-1,2\}\biggr\}, \;\mbox{and}\;
\biggl\{\{-2,-2\},\{-1,1\}\biggr\}.
$$
So each element of $X(2)$ is the pairing of $\{-2,-1,1,2\}$ into
(two) unordered pairs.

 We have
\begin{equation}
|X(n)|=\frac{|S(2n)|}{|H(n)|}=\frac{(2n)!}{2^nn!}=1\cdot
3\cdot\ldots \cdot (2n-1). \nonumber
\end{equation}
\subsection{Canonical projections $p_{n,n+1}:X(n+1)\rightarrow X(n)$. The projective limit of the spaces $X(n)$}
\label{SubsectionSpaceX}

Given an element $\breve{x}'\in X(n+1)$ we define its derivative
element $\bx\in X(n)$ as follows. Represent $\breve{x}'$ as $n+1$
unordered pairs, as it is explained in the previous Section. If
$n+1$ and $-n-1$ are in the same pair, then $\breve{x}$ is
obtained from $\breve{x}'$ by deleting this pair. Suppose that
$n+1$ and $-n-1$ are in different pairs. Then $\breve{x}'$ can be
written as
$$
\breve{x}'=\biggl\{\{i_1,i_2\},\ldots, \{i_m,-n-1\},\ldots,
\{i_k,n+1\},\ldots, \{i_{2n+1},i_{2n+2}\}\biggr\}.
$$
In this case $\breve{x}$ is obtained from $\breve{x}'$ by removing
$-n-1$, $n+1$  from pairs $\{i_m,-n-1\}$ and $\{i_k, n+1\}$
correspondingly, and by replacing two these pairs, $\{i_m,-n-1\}$
and $\{i_k, n+1\}$, by one pair $\{i_m,i_k\}$.  The map
$\breve{x}'\rightarrow \breve{x}$, denoted by $p_{n,n+1}$, will be
referred to as the canonical projection of $X(n+1)$ onto $X(n)$.

Consider the sequence
$$
X(1)\leftarrow\ldots\leftarrow X(n)\leftarrow
X(n+1)\leftarrow\ldots
$$
of canonical projections, and let
$$
X=\varprojlim X(n)
$$
denote the projective limit of the sets $X(n)$. By definition, the
elements of $X$ are arbitrary sequences $\breve{x}=(\breve{x}_1,
\breve{x}_2,\ldots )$, such that $\breve{x}_n\in X(n)$, and
$p_{n,n+1}(\breve{x}_{n+1})=\breve{x}_n$. The set $X$ is a closed
subset of the compact space of all sequences $(\breve{x}_n)$,
therefore, it is a compact space itself.

In what follows we denote by $p_n$ the projection $X\rightarrow
X(n)$ defined by $p_n(\breve{x})=\breve{x}_n$.

\subsection{Cycles.  Representation of elements of $X(n)$ in terms of
arrow configurations on circles}\label{SectionArrows} Let $\bx$ be
an element of $X(n)$. Then $\bx$ can be identified with  arrow
configurations on  circles. Such  arrow configurations can be
constructed as follows. Once $\bx$ is written as a collection of
$n$ unordered pairs, one can represent $\bx$ as a union of cycles
of the form
\begin{equation}\label{formcycle}
j_1\rightarrow -j_2\rightarrow j_2\rightarrow -j_3\rightarrow
j_3\rightarrow\ldots\rightarrow -j_k \rightarrow j_k \rightarrow
-j_1\rightarrow j_1,
\end{equation}
where $j_1,j_2,\ldots ,j_k$ are distinct integers from the set
$\{-n,\ldots ,n\}$. For example, take
\begin{equation}\label{element}
\breve{x}=\biggl\{\{1,3\},\{-2,5\}, \{2,-1\},
\{-3,-5\},\{4,-6\},\{-4,6\}\biggr\}.
\end{equation}
Then $\bx\in X(3)$, and it is possible to think about $\bx$ as a
union of two cycles, namely
\begin{equation}\label{circle1}
1\rightarrow 3\rightarrow-3\rightarrow-5\rightarrow
5\rightarrow-2\rightarrow 2\rightarrow-1\rightarrow 1, \nonumber
\end{equation}
and
\begin{equation}\label{circle2}
4\rightarrow-6\rightarrow 6\rightarrow -4\rightarrow 4. \nonumber
\end{equation}
Cycle  (\ref{formcycle}) can be represented as a circle with
attached arrows. Namely,  we put on a circle points labelled by
$|j_1|$, $|j_2|$,$\ldots$, $|j_k|$, and attach arrows to these
points according to the following rules. The arrow attached to
$|j_1|$ is directed clockwise. If the next integer in the cycle
(\ref{formcycle}), $j_2$, has the same sign as $j_1$, then the
direction of the arrow attached to $|j_2|$ is the same as the
direction of the arrow attached to $|j_1|$, i.e. clockwise.
Otherwise, if the sign of $j_2$ is opposite to the sign of $j_1$,
the direction of the arrow attached to $|j_2|$ is opposite to the
direction of the arrow attached to $|j_1|$, i.e. counterclockwise.
Next, if the integer $j_3$ has the same sign as $j_2$, then the
direction of the arrow attached to $|j_3|$ is the same as the
direction of the arrow attached to $|j_2|$, etc. For example,  the
representation of of the element $\bx$ defined by (\ref{element})
in terms of arrow configurations on circles is shown on Fig. 1.
\begin{figure}
\begin{picture}(100,150)
\put(-10,100){\circle{200}}
\put(100,100){\circle{200}}
\put(-10,120){\circle*{2}} \put(-10,80){\circle*{2}}
\put(-30,100){\circle*{2}} \put(10,100){\circle*{2}}
\put(-10,124){$1$} \put(-10,69){$5$} \put(-42,100){$2$}
\put(14,100){$3$}
\put(-10,120){\vector(1,0){10}} \put(-10,80){\vector(-1,0){10}}
\put(-30,100){\vector(0,1){10}} \put(10,100){\vector(0,1){10}}
\put(100,120){\circle*{2}} \put(100,80){\circle*{2}}
\put(100,120){\vector(1,0){10}} \put(100,80){\vector(-1,0){10}}
\put(100,124){$4$} \put(100,69){$6$}
\end{picture}
\caption{The representation of the element
$$
\breve{x}=\biggl\{\{1,3\},\{-2,5\}, \{2,-1\},
\{-3,-5\},\{4,-6\},\{-4,6\}\biggr\}
$$
in terms of  arrow configurations on circles. The first circle
(from the left) represents  cycle $1\rightarrow
3\rightarrow-3\rightarrow-5\rightarrow 5\rightarrow-2\rightarrow
2\rightarrow-1\rightarrow 1$, and the second circle represents
cycle $4\rightarrow-6\rightarrow 6\rightarrow -4\rightarrow 4$.}
\end{figure}

In this representation the projection $p_{n,n+1}:X(n+1)\rightarrow
X(n)$ is reduced to removing the point $n+1$ together with the
attached arrow.
\section{The $t$-measures on $X$}
\subsection{Probability measures $\mu_t^{(n)}$ on $X(n)$, and $\mu_t$ on $X$}

The measures $\mu_t^{(n)}$ on the spaces $X(n)$ are natural
analogues of the Ewens measures on the group $S(n)$ described in
Kerov, Olshanski and Vershik \cite{KOV2}.

\begin{defn}\label{DefinitionEwensMeasures} For $t>0$ we set
$$
\mu_t^{(n)}(\breve{x})=\frac{t^{[\breve{x}]_n}}{t(t+2)\ldots
(t+2n-2)},
$$
where $\breve{x}\in X(n)$, and $[\breve{x}]_n$ denotes the number
of cycles in $\breve{x}$, or the number of circles in the
representation of $\bx$ in terms of arrow configurations, see
Section \ref{SectionArrows} .
\end{defn}
\begin{prop}\label{PROPOSITION4.2}
a) We have
\begin{equation}\label{normtequation}
\sum\limits_{\breve{x}\in X(n)}\mu_t^{(n)}(\breve{x})=1.
\end{equation}
Thus $\mu_t^{(n)}(\breve{x})$ can be understood as a probability
measure
on $X(n)$. \\
b) Given $t>0$, the canonical projections $p_{n,n+1}$ preserve the
measures $\mu_t^{(n)}(\breve{x})$, which means that the condition
\begin{equation}\label{mutnproperty}
\mu_t^{(n+1)}\biggl(\{\breve{x}'\;\vert\; \breve{x}'\in
X(n+1),p_{n,n+1}(\breve{x}')=\breve{x}\}\biggr)=\mu_t^{(n)}(\breve{x})
\end{equation}
is satisfied for each $ \breve{x}\in X(n)$.
\end{prop}
\begin{proof}
If $n=1$, then $X(1)$ consists of only one element, namely
$\{-1,1\}$, and from Definition \ref{DefinitionEwensMeasures} we
immediately see that equation (\ref{normtequation}) is satisfied
in this case.

Let $\breve{x}$ be an arbitrary element of $X(n)$. Represent
$\breve{x}$ in terms of circles with attached arrows, as it is
explained in Section \ref{SectionArrows}. Consider the set
\begin{equation}\label{set}
\left\{\breve{x}'|\; \breve{x}'\in X(n+1),
p_{n,n+1}(\breve{x}')=\breve{x}\right\}.
\end{equation}
It is not hard to see that this set consists of $2n+1$ points.
Indeed, given $\breve{x}\in X(n)$ we  can obtain $\breve{x}'$ from
set (\ref{set}) (i.e. $\breve{x}'$ which lies above $\breve{x}$
with respect to the canonical projection $p_{n,n+1}$)  by adding
an arrow to existing circle in $2n$ ways, or by creating a new
circle.

If $\breve{x}'$ is obtained from $\breve{x}$ by creating a new
circle, then
$$
[\breve{x}']_{n+1}=[\breve{x}]_n+1,\;\; \mbox{and}\;\;
t^{[\breve{x}']_{n+1}}=t^{[\breve{x}]_{n}+1}.
$$
If $\breve{x}'$ is obtained from $\breve{x}$ by adding an arrow to
an existing circle, then
$$
[\breve{x}']_{n+1}=[\breve{x}]_n.
$$
Therefore, the relation
$$
\sum\limits_{\breve{x}'\in
X(n+1)}t^{[\breve{x}']_{n+1}}=(t+2n)\sum\limits_{\breve{x}\in
X(n)}t^{[\breve{x}]_{n}}
$$
is satisfied. From the recurrent relation above we obtain
$$
\sum\limits_{\breve{x}'\in
X(n+1)}t^{[\breve{x}']_{n+1}}=t(t+2)\ldots (t+2n).
$$
This formula is equivalent to equation (\ref{normtequation}), and
the first statement of the Proposition is proved.

Let us now prove the second statement of the Proposition. We need
to show that the condition (\ref{mutnproperty}) is satisfied for
each $\breve{x}\in X(n)$. We have
\begin{equation}\label{z1}
\mu_t^{(n+1)}\biggl(\{\breve{x}'\;\vert\; \breve{x}'\in
X(n),p_{n,n+1}(\breve{x}')=\breve{x}\}\biggr)=\sum\limits_{\breve{x}':\breve{x}'\in
X(n+1),
p_{n,n+1}(\breve{x}')=\breve{x}}\frac{t^{[\breve{x}']_{n+1}}}{t(t+2)\ldots
(t+2n)},
\end{equation}
where we have used the definition of $\mu_t^{(n)}$, Definition
\ref{DefinitionEwensMeasures}. By the same argument as in the
proof of the first statement of the Proposition the sum in the
righthand side of equation (\ref{z1}) can be decomposed into two
sums. This first sum runs over those $\breve{x}'$ that are
obtained from $\breve{x}$ by adding an arrow to one of the
existing circles of $\breve{x}$. This sum is equal to
\begin{equation}\label{z2}
\frac{(2n)t^{[\breve{x}]_{n}}}{t(t+2)\ldots (t+2n)}.
\end{equation}
The second sum runs over those $\breve{x}'$ that are obtained from
$\breve{x}$ by creating a new circle. There is only one such
$\breve{x}'$, and its contribution is
\begin{equation}\label{z3}
\frac{t\; t^{[\breve{x}]_{n}}}{t(t+2)\ldots (t+2n)}.
\end{equation}
Adding expressions (\ref{z2}) and (\ref{z3}) we obtain
$$
\mu_t^{(n+1)}\biggl(\{\breve{x}'\;\vert\; \breve{x}'\in
X(n),p_{n,n+1}(\breve{x}')=\breve{x}\}\biggr)=\frac{(2n)t^{[\breve{x}]_{n}}}{t(t+2)\ldots
(t+2n)}+\frac{t\; t^{[\breve{x}]_{n}}}{t(t+2)\ldots (t+2n)},
$$
and the righthand side of the above equation  is
$\mu_t^{(n)}(\breve{x})$.
\end{proof}
It follows from Proposition \ref{PROPOSITION4.2} that for any
given $t>0$, the canonical projection $p_{n-1,n}$ preserves the
measures $\mu_t^{(n)}$. Hence the measure
$$
\mu_t=\varprojlim \mu_t^{(n)}
$$
on $X$ is correctly defined, and it is a probability measure. Note
that as in the case considered in Kerov, Olshanski, and Vershik
\cite{KOV2}, Section 2, the probability space $(X,\mu_t)$ is
closely related to the Chines Restaurant Process construction, see
Aldous \cite{Aldous}, Pitman \cite{Pitman}.
\subsection{The group $S(2\infty)$ and its action on the space $X$}
First we describe the right action of the group $S(2n)$ on the
space $X(n)$, and then we extend it to the right action of
$S(2\infty)$ on $X$.

Let $\breve{x}_n\in X(n)$. Then $\breve{x}_n$ can be written as a
collection of $n$ unordered pairs (equation
(\ref{representationofelement})). Let  $g$ be a permutation from
$S(2n)$,
$$
g:\;\; \left(\begin{array}{ccccc}
  -n & -n+1 & \ldots & n-1 & n \\
  g(-n) & g(-n+1) & \ldots & g(n-1) & g(n) \\
\end{array}\right).
$$
The right action of the group $S(2n)$ on the space $X(n)$ is
defined by
$$
\breve{x}_n\cdot
g=\biggl\{\{g(i_1),g(i_2)\},\{g(i_3),g(i_4)\},\ldots,
\{g(i_{2n-1}),g(i_{2n})\}\biggr\}.
$$
\begin{prop}The canonical projection $p_{n,n+1}$ is equivariant
with respect to the right action of the group $S(2n)$ on the space
$X(n)$, which means
$$
p_{n,n+1}(\breve{x}\cdot g)=p_{n,n+1}(\breve{x})\cdot g,
$$
for all $\breve{x}\in X(n+1)$, and all $g\in S(2n)$.
\end{prop}
\begin{proof}
Let $\breve{x}$ be an arbitrary element of $X(n+1)$. Represent
$\breve{x}\in X(n+1)$ in terms of configurations of arrows on
circles, as it is described in Section \ref{SectionArrows}. In
this representation the right action of an element $g$ from
$S(2n)$ on $x$ is reduced to permutations of numbers $1,2,\ldots
n$ on the circles, and to changes in the directions of the arrows
attached to these numbers. The number $n+1$, and the direction of
the arrow attached to $n+1$ remains unaffected by the action of
$S(2n)$. Since $p_{n,n+1}(\breve{x})$ is obtained from $\breve{x}$
by deleting $n+1$ together with the attached arrow, the statement
of the Proposition follows.
\end{proof}
Since the canonical projection $p_{n,n+1}$ is equivariant, the
right action of $S(2n)$ on $X(n)$ can be extended to the right
action of $S(2\infty)$ on $X$. For $n=1,2,\ldots $ we identify
$S(2n)$ with the subgroup of permutations $g\in S(2n+2)$
preserving the elements $-n-1$ and $n+1$ of the set
$\{-n-1,-n,\ldots,-1,1,\ldots,n,n+1\}$, i.e.
$$
S(2n)=\biggl\{g\biggl|g\in S(2n+2),\; g(-n-1)=-n-1,\; \mbox{and}\;
g(n+1)=n+1 \biggr\}.
$$

Let $S(2)\subset S(4)\subset S(4)\ldots $ be the collection of
such subgroups. Set
$$
S(2\infty)=\bigcup_{n=1}^{\infty}S(2n).
$$
Thus $S(2\infty)$ is the inductive limit of subgroups $S(2n)$,
$$
S(2\infty)=\varinjlim S(2n).
$$
If $\breve{x}=(\breve{x}_1,\breve{x}_2,\ldots )\in X$, and $g\in
S(2\infty)$, then the right action of $S(2\infty)$ on
$X=\varprojlim X(n)$,
$$X\times
S(2\infty)\longrightarrow X,
$$ is defined as $\breve{x}\cdot g=\check{y}$, where
$\breve{x}_n\cdot g=\breve{y}_n$ for all $n$ so large that $g\in
S(2\infty)$ lies in $S(2n)$.
\begin{prop}\label{Proposition1.444444}We have
$$
p_n(\breve{x}\cdot g)=p_n(\breve{x})\cdot g
$$
for all $\breve{x}\in X$, $g\in S(2\infty)$, and for all $n$ so
large that $g\in S(2n)$.
\end{prop}
\begin{proof}The claim follows immediately from the very
definition of the projection $p_n$, and of the right action of
$S(2\infty)$ on $X$.
\end{proof}
\subsection{The fundamental cocycle}\label{SectionCocycle}
 Recall that $[.]_n$ denotes
the number of cycles in the cycle representation of an element
from $X(n)$ (see Section \ref{SectionArrows} where the cycle
structure of the elements from $X(n)$ was introduced).
\begin{prop}
For any $\breve{x}=(\breve{x}_n)\in X$, and $g\in S(2\infty)$, the
quantity
$$
c(\breve{x};g)=[p_n(\breve{x}\cdot
g)]_n-[p_n(\breve{x})]_n=[p_n(\breve{x})\cdot g]_n-[p_n(\bx)]_n
$$
does not depend on $n$ provided that $n$ is so large that $g\in
S(2n)$.
\end{prop}
\begin{proof} Let $g$ be an element of $S(2n)$. To prove the
Proposition it is enough to show that the condition
\begin{equation}\label{pcompatibility}
[p_{n,n+1}(\bx)\cdot g]_n-[p_{n,n+1}(\bx)]_n=[\bx\cdot
g]_{n+1}-[\bx]_{n+1}
\end{equation}
is satisfied for any element $\bx$ of $X(n+1)$. Since $g\in S(2n)$
can be always represented as a product of transpositions, and
since $p_{n,n+1}$ is equivariant with respect to the right action
of the group $S(2n)$, it is enough to prove (\ref{pcompatibility})
for the case when $g$ is a transposition. Thus we assume that $g$
is a transposition $(ij)\in S(2n)$, where $i$ and $j$ are two
different elements of the set $\{-n,\ldots,-1,1,\ldots,n\}$.

Let $\bx$ be an element of $X(n+1)$. Write $\bx$ as a collection
of cycles as it is explained in Section \ref{SectionArrows}.
Assume that both $i$ and $j$ belong to the same cycle of $\bx$. We
check that the multiplication of $\bx$ by $(i,j)$ from the right
either splits this cycle into two, or transforms it into a
different cycle. Thus we have
$$
[\bx\cdot g]_{n+1}-[\bx]_{n+1}=1\;\;\mbox{or}\;\; 0.
$$
The value of the difference $[\bx\cdot g]_{n+1}-[\bx]_{n+1}$
depends on the mutual configuration of $-i, i, -j$, and $j$ in the
cycle containing $i, j$. More explicitly, if the pair with $-i$ is
situated from the left to the pair with $i$, and, at the same
time, the pair with $-j$ is situated from the left to the pair
with $j$, then the value of $[\bx\cdot g]_{n+1}-[\bx]_{n+1}$ is
$1$. In this case the cycle under considerations has the form
$$
k_1\rightarrow\ldots \rightarrow k_{m}\rightarrow -i \rightarrow
i\rightarrow-k_{m+1}\rightarrow\ldots\rightarrow
k_{p}\rightarrow-j\rightarrow j\rightarrow -k_{p+1}\rightarrow
\ldots \rightarrow -k_1\rightarrow k_1,
$$
or the form
$$
k_1\rightarrow\ldots \rightarrow k_{m}\rightarrow -j \rightarrow
j\rightarrow-k_{m+1}\rightarrow\ldots\rightarrow
k_{p}\rightarrow-i\rightarrow i\rightarrow -k_{p+1}\rightarrow
\ldots \rightarrow -k_1\rightarrow k_1,
$$
and the corresponding mutual configuration of $-i, i, -j$, and $j$
is
$$
\{.,-i\}\{i,.\}...\{.,-j\}\{j,.\},
$$
or
$$
\{.,-j\}\{j,.\}...\{.,-i\}\{i,.\}.
$$
If in the cycle under considerations the pair with $-i$ stands
from the right to the pair with $i$, and, at the same time, the
pair with $-j$ is situated from the right to the pair with $j$,
then the value of $[\bx\cdot g]_{n+1}-[\bx]_{n+1}$ is $1$ as well.
In this case the mutual configuration of $-i, i, -j$, and $j$ is
$$
\{.,i\}\{-i,.\}...\{.,j\}\{-j,.\},
$$
or
$$
\{.,j\}\{-j,.\}...\{.,i\}\{-i,.\}.
$$

Otherwise, if the mutual configuration of $-i, i, -j$, and $j$ is
different from those described above, then $[\bx\cdot
g]_{n+1}-[\bx]_{n+1}=0$.

On the other hand, the numbers $i$ and $j$ belong to the one and
the same cycle of $p_{n,n+1}(\bx)$ if and only if they belong to
one and the same cycle of $\bx$. Moreover, if $i$ and $j$ belong
to the same cycle of $\bx$, then the mutual configuration of $-i,
i, -j$, and $j$ is the same as in $p_{n,n+1}(\bx)$. Thus we
conclude that equation (\ref{pcompatibility}) holds true if $i$
and $j$ belong to the same cycle of $\bx$.

If $i$ and $j$ belong to different cycles then the two cycles of
$\bx$ containing the elements $i$ and $j$ merge into a single
cycle of the product $\bx\cdot (ij)$, and we clearly have
$$
[\bx\cdot (ij)]_{n+1}-[\bx]_{n+1}=-1.
$$
The same equation holds true if we replace $\bx$ by
$p_{n,n+1}(\bx)$, so equation (\ref{pcompatibility}) holds true
when $i$ and $j$ belong to different cycles as well.
\end{proof}
\subsection{Quasiinvariance of $\mu_t$}
\begin{prop} Each of measures $\mu_t$, $0<t<\infty$, is
quasiinvariant with respect to the action of $S(2\infty)$ on the
space $X=\varprojlim X(n)$. More precisely,
$$
\frac{\mu_t(d\bx\cdot g)}{\mu_t(d\bx)}=t^{c(\bx;g)};\;\;\bx\in
X,\; g\in S(2\infty),
$$
where $c(\bx;g)$ is the fundamental cocycle of Section
\ref{SectionCocycle}.
\end{prop}
\begin{proof}
We need to check that
\begin{equation}\label{1.7-1.7.2}
\mu_t(V\cdot g)=\int_{V}t^{c(\bx;g)}\mu_t(d\bx),\;\; g\in
S(2\infty)
\end{equation}
for every Borel subset $V\subseteq X$. Choose $m$ so large that
$g\in S(2m)$, and let $n\geq m$. Take $\by\in X(n)$, and set
$V_n(\by)=p_n^{-1}(\by)\subset X$. This is a cylinder set. It is
enough to check equation (\ref{1.7-1.7.2}) for $V=V_n(\by)$.  Note
that $V_n(\by)\cdot g=V_n(\by\cdot g)$. This follows from the fact
that the projection $p_n$ is equivariant with respect to the right
action of the group, see Proposition \ref{Proposition1.444444}.
From the definition of $\mu_t$ we conclude that
$\mu_t(V_n(\by))=\mu_t^{(n)}(\{\by\})$, hence
$$
\mu_t\left(V_n(g)\cdot g\right)=\mu_t^{(n)}(\{\by\cdot g\}).
$$
On the other hand,
$$
c(\bx;g)=[p_n(\bx\cdot g)]_n-[p_n(\bx)]_n=[\by\cdot g]_n-[\by]_n
$$
for all $\bx\in V_n(y)$. Therefore, equation (\ref{1.7-1.7.2})
takes the form
$$
\mu_t^{(n)}\left(\left\{\by\cdot g\right\}\right)=t^{[\by\cdot
g]_n-[\by]_n}\mu_t^{(n)}\left(\{\by\}\right).
$$
Using the very definition of $\mu_t^{(n)}$ we check that the
equation just written above holds true. Therefore, equation
(\ref{1.7-1.7.2}) holds true as well.
\end{proof}
\section{The representations $T_{z,\frac{1}{2}}$}

The aim of this Section is to introduce a family
$T_{z,\frac{1}{2}}$ of unitary representations of the group
$S(2\infty)$. These representations are parameterized by points
$z\in \C\setminus\{0\}$, and can be viewed as the analogues of the
generalized regular representations introduced in Kerov,
Olshanski, and Vershik \cite{KOV1, KOV2}. As in the case of the
generalized regular representations, each element of the family
$T_{z,\frac{1}{2}}$ can be approximated by the regular
representation of the group $S(2n)$. This enables us to give an
explicit formula for the restriction of the spherical function of
the representation $T_{z,\frac{1}{2}}$ to $S(2n)$, and to
introduce the measures on Young diagrams associated with
representations $T_{z,\frac{1}{2}}$. Then it will be shown that
these measures can be understood as the $z$-measures with the Jack
parameter $\theta=\frac{1}{2}$ in the notation of Section
\ref{Sectionztheta}. Thus the $z$-measures with the Jack parameter
$\theta=\frac{1}{2}$ will be associated to representations
$T_{z,\frac{1}{2}}$ in a similar way as the $z$-measures with the
Jack parameter $\theta=1$ are associated with generalized regular
representations in Kerov, Olshanski, and Vershik \cite{KOV2},
Section 4.
\subsection{Definition of $T_{z,\frac{1}{2}}$}

Let $(\X,\Sigma,\mu)$ be a measurable space. Let $G$ be a group
which acts on $\X$ from the right, and preserves the Borel
structure. Assume that the measure $\mu$ is quasiinvariant, i.e.
the condition
$$
d\mu(\bx\cdot g)=\delta(\bx;g)d\mu(\bx)
$$
is satisfied for some nonnegative $\mu$-integrable function
$\delta(\bx;g)$ on $\X$, and for every $g$, $g\in G$. Set
\begin{equation}\label{3.1-3.1}
\left(T(g) f\right)(\bx)=\tau(\bx;g)f(\bx\cdot g),\; f\in
L^{2}(\X,\mu),
\end{equation}
where $|\tau(\bx;g)|^2=\delta(\bx;g)$. If
$$\tau(\bx;g_1g_2)=\tau(\bx\cdot
g_1;g_2)\tau(\bx;g_1),\; \bx\in\X, g_1,g_2\in G,
$$
then equation (\ref{3.1-3.1}) defines a unitary representation $T$
of $G$ acting in the Hilbert space $L^2(\X;\mu)$. The function
$\tau(\bx;g)$ is called a multiplicative cocycle.

Let $z\in\C$ be a nonzero complex number. We apply the general
construction described above for the space $\X=X$, the group
$G=S(2\infty)$, the measure $\mu=\mu_t$ (where $t=|z|^2$), and the
cocycle $\tau(\bx;g)=z^{c(\bx;g)}$. In this way we get a unitary
representation of $S(2\infty)$, $T_{z,\frac{1}{2}}$, acting in the
Hilbert space $L^2(X,\mu_t)$ according to the formula
$$
\left(T_{z,\frac{1}{2}}(g)f\right)(\bx)=z^{c(\bx;g)}f(\bx\cdot
g),\; f\in L^2(X,\mu_t),\;\bx\in X,\; g\in S(2\infty).
$$
\subsection{Approximation by quasi-regular representations}

\begin{defn} For $n=1,2,\ldots $ let $\mu_1^{(n)}$ denote the
normalized Haar measure on $X(n)$. The regular representation
$Reg^n$ of the group $S(2n)$ acting in the Hilbert space
$L^{(2)}(X(n),\mu_1^{(n)})$ is defined by
$$
\left(Reg^n(g)f\right)(\bx)=f(\bx\cdot g),\; \bx\in X(n),\; g\in
S(2n),\; f\in L^2(X(n),\mu_t).
$$
\end{defn}
\begin{prop}\label{Proposition3.3-3.2} The representations $Reg^n$ and
$T_{z,\frac{1}{2}}|_{L^{2}(X(n),\mu_t^{(n)})}$ of $S(2n)$ are
equivalent.
\end{prop}
\begin{proof} Set
\begin{equation}\label{3.3-3.2.1}
F_z^{(n)}(\bx)=\left(\frac{1\cdot 3\cdot \ldots \cdot
(2n-1)}{t\cdot (t+2)\cdot\ldots\cdot(t+2n-2)
}\right)^{1/2}z^{[\bx]_n},\; \bx\in X(n),
\end{equation}
 and denote by $f_z^{(n)}$ the operator
of multiplication by $F_z^{(n)}$. Since
$$
|F_z^{(n)}(\bx)|^2=\frac{\mu_t^{(n)}(\bx)}{\mu_1^{(n)}(\bx)},
$$
the operator $f_z^{(n)}$ carries $L^2(X(n),\mu_t^{(n)})$ onto
$L^2(X(n),\mu_1^{(n)})$, and defines an isometry. Moreover, it is
straightforward to check that $f_z^{(n)}$ intertwines for the
$S(2n)$-representations $Reg^n$ and
$T_{z,\frac{1}{2}}|_{L^{2}(X(n),\mu_t^{(n)})}$.
\end{proof}
Next we need the notion of the inductive limits of
representations. Let $G(1)\subseteq G(2)\subseteq\ldots $ be a
collection of finite groups, and set $G=\bigcup_{n=1}^{\infty}
G(n)$. Thus $G$ is the inductive limit of the groups $G(n)$.
Assume that for each $n$ a unitary representation $T_n$ of $G(n)$
is defined. Denote by $H(T_n)$ the Hilbert space in which the
representation $T_n$ acts, and denote by $H$ the Hilbert
completion of the space $\bigcup_{n=1}^{\infty}H(T_n)$. We also
assume that an isometric embedding $\alpha_n: H(T_n)\rightarrow
H(T_{n+1})$ is given, and that this embedding is intertwining for
the $G(n)$-representations $T_n$ and $T_{n+1}|_{G(n)}$.
\begin{defn}
A unitary representation $T$ of the group $G$ acting in the
Hilbert space $H$, and uniquely defined by
$$
T(g)\xi=T_n(g)\xi,\;\mbox{if}\; g\in G(n)\;\mbox{and}\;\xi\in
H(T_n)
$$
is called the inductive limit of representations $\{T_n\}$.
\end{defn}
Consider the following diagram
$$
\begin{array}{cccccccccc}
H(T_1)&\stackrel{f_1}{\longrightarrow}& H(T_2) &
\stackrel{f_2}{\longrightarrow}& H(T_3)
&\stackrel{f_3}{\longrightarrow}&\ldots\\
\vert &&\vert &&\vert \\
\vert\lefteqn{F_1}&& \vert\lefteqn{F_2}&& \vert\lefteqn{F_3}\\
\downarrow && \downarrow && \downarrow\\
H(S_1)&\stackrel{\rho_1}{\longrightarrow}& H(S_2) &
\stackrel{\rho_2}{\longrightarrow}& H(S_3)
&\stackrel{\rho_3}{\longrightarrow}&\ldots
\end{array}
$$
Here $\{T_n\}_{n=1}^{\infty}$ and $\{S_n\}_{n=1}^{\infty}$ are
collections of representations of $G(1), G(2),\ldots $, and $
G(1)\subseteq G(2)\subseteq\ldots $. The following  fact is almost
obvious, and we formulate it as a Proposition without proof.
\begin{prop}\label{PropositionA1}
Assume that for each $n=1,2,\ldots $ the following conditions are
satisfied\begin{itemize}
    \item The linear map $F_n$ is from $H(T_n)$ onto $H(S_n)$, which is intertwining for $T_n$
    and $S_n$.
    \item The linear map $f_n$ is an isometric embedding of
    $H(T_n)$ into $H(T_{n+1})$, which is intertwining for the
    $G(n)$-representations $T_n$ and $T_{n+1}|_{G(n)}$.
    \item The map $\rho_n$ is an isometric embedding of $H(S_n)$
    into $H(S_{n+1})$ such that the diagram
    $$
\begin{array}{ccc}
H(T_n)&\stackrel{f_n}{\longrightarrow}& H(T_n)\\
\vert  &&\vert \\
\vert\lefteqn{F_n}&& \vert\lefteqn{F_{n+1}}\\
\downarrow &&  \downarrow\\
H(S_n)&\stackrel{\rho_n}{\longrightarrow}& H(S_{n+1})
\end{array}
$$
is commutative, i.e., the condition
$f_n=F_{n+1}^{-1}\circ\rho_n\circ F_n$ holds true.
\end{itemize}
Then the inductive limits of $\{T_n\}_{n=1}^{\infty}$, and of
$\{S_n\}_{n=1}^{\infty}$ are well-defined, and these inductive
limits are equivalent.
\end{prop}
\begin{prop}
Define the operator $L_z^{(n)}$,
$$ L_z^{(n)}:
L^{2}(X(n),\mu_1^{(n)})\longrightarrow L^{2}(X(n+1),\mu_1^{(n)})
$$
as follows: if $f\in L^{2}(X(n),\mu_1^{(n)})$, and $\bx\in
X(n+1)$, then
\begin{equation}\label{Lz}
\left(L_z^{(n)}f\right)(\bx)=
\left\{%
\begin{array}{lll}
    z\sqrt{\frac{2n+1}{2n+t}}f(\bx), & \bx\in X(n)\subset X(n+1), \\
    \\
    \sqrt{\frac{2n+1}{2n+t}}f(p_{n,n+1}(\bx)), & \bx\in X(n+1)\setminus X(n).\\
\end{array}%
\right.
\end{equation}
For any nonzero complex number $z$ the operator $L_z^{(n)}$
provides an isometric embedding
$L^{2}(X(n),\mu_1^{(n)})\longrightarrow L^{2}(X(n+1),\mu_1^{(n)})$
which intertwines for the $S(2n)$-representations $Reg^n$ and
$Reg^{n+1}|_{S(2n)}$. Let $T_{z,\frac{1}{2}}'$ denote the
inductive limit of the representations $Reg^n$ with respect to the
embedding
$$
\begin{array}{ccccc}
L^{2}(X(1),\mu_1^{(1)})&\stackrel{L_z^{(1)}}{\longrightarrow}&
L^{2}(X(2),\mu_2^{(2)})&\stackrel{L_z^{(2)}}{\longrightarrow}&\ldots
\end{array}
$$
Then the representations $T_{z,\frac{1}{2}}'$ and
$T_{z,\frac{1}{2}}$ are equivalent.
\end{prop}
\begin{proof}
For $f\in L^{2}(X(n),\mu_t^{(n)})$, and $\bx\in X(n+1)$ set
$$
\left(\alpha^{(n)}f\right)(\bx)=f(p_{n,n+1}(\bx)).
$$
Then $\alpha^{(n)}$ is an isometric embedding of
$L^{2}(X(n),\mu_t^{(n)})$ into $L^{2}(X(n+1),\mu_t^{(n+1)})$.
Using the definition of the representation $T_{z,\frac{1}{2}}$ it
is straightforward to verify that $\alpha^{(n)}$ intertwines for
the $S(2n)$-representations
$T_{z,\frac{1}{2}}|_{L^{2}(X(n),\mu_t^{(n)})}$ and
$T_{z,\frac{1}{2}}|_{L^{2}(X(n+1),\mu_t^{(n+1)})}$. This enables
us to consider $T_{z,\frac{1}{2}}$ as the inductive limit of
$S(2n)$-representations of
$T_{z,\frac{1}{2}}|_{L^{2}(X(n),\mu_t^{(n)})}$. Now examine the
following diagram
$$
\begin{array}{cccccccccc}
L^{2}(X(1),\mu_t^{(1)})&\stackrel{\alpha^{(1)}}{\longrightarrow}&
L^{2}(X(2),\mu_t^{(2)}) &
\stackrel{\alpha^{(2)}}{\longrightarrow}& L^{2}(X(3),\mu_t^{(3)})
&\stackrel{\alpha^{(3)}}{\longrightarrow} &\ldots\\
\vert &&\vert &&\vert \\
\vert\lefteqn{f_z^{(1)}}&& \vert\lefteqn{f_z^{(2)}}&& \vert\lefteqn{f_z^{(3)}}\\
\downarrow && \downarrow && \downarrow\\
L^{2}(X(1),\mu_1^{(1)})&\stackrel{L^{(1)}_z}{\longrightarrow}&
L^{2}(X(2),\mu_1^{(2)}) & \stackrel{L^{(2)}_z}{\longrightarrow}&
L^{2}(X(3),\mu_1^{(3)}) &\stackrel{L^{(3)}_z}{\longrightarrow}
&\ldots
\end{array}
$$
where the operators $f_z^{(n)}$ are that of multiplications by
$F_z^{(n)}$ introduced in the proof of Proposition
\ref{Proposition3.3-3.2}. Recall that $f_z^{(n)}$ intertwines for
the $S(2n)$-representations $Reg^n$ and
$T_{z,\frac{1}{2}}|_{L^{2}(X(n),\mu_t^{(n)})}$. We determine
$L_z^{(n)}$ from the condition of commutativity of the diagram
$$
\begin{array}{ccc}
L^{2}(X(n),\mu_t^{(n)}) &
\stackrel{\alpha^{(n)}}{\longrightarrow}&
L^{2}(X(n+1)),\mu_t^{(n+1)})\\
\vert &&\vert \\
\vert\lefteqn{f_z^{(n)}}&& \vert\lefteqn{f_z^{(n)}}\\
\downarrow  && \downarrow\\
L^{2}(X(n),\mu_1^{(n)})&\stackrel{L^{(n+1)}_z}{\longrightarrow}&
L^{2}(X(n+1),\mu_1^{(n+1)})
\end{array}
$$
and obtain that $L^{(n)}_z$ is given by formula (\ref{Lz}).
Moreover, from equation (\ref{Lz}) we see that $L_z^{(n)}$ defines
the isometric embedding of $L^{2}(X(n),\mu_t^{(n)})$ into
$L^{2}(X(n+1),\mu_t^{(n+1)})$. Now we use Proposition
\ref{PropositionA1}  to conclude that the inductive limit
$T_{z,\frac{1}{2}}'$ of the representations $Reg^n$ with respect
to the embedding
$$
\begin{array}{ccccc}
L^{2}(X(1),\mu_1^{(1)})&\stackrel{L_z^{(1)}}{\longrightarrow}&
L^{2}(X(2),\mu_2^{(2)})&\stackrel{L_z^{(2)}}{\longrightarrow}&\ldots
\end{array}
$$
is well-defined, and it is equivalent to $T_{z,\frac{1}{2}}$.
\end{proof}
\subsection{A formula for the spherical function of $T_{z,\frac{1}{2}}$}

Let $(G,K)$ be a Gelfand pair, and let $T$ be a unitary
representation of $G$ acting in the Hilbert space $H(T)$. Assume
that $\xi$ is a unit vector in $H(T)$ such that $\xi$ is
$K$-invariant, and such that the span of vectors of the form
$T(g)\xi$ (where $g\in G$) is dense in $H(T)$. In this case $\xi$
is called the spherical vector, and the matrix coefficient
$(T(g)\xi,\xi)$ is called the spherical function of the
representation $T$. Two spherical representations are equivalent
if and only if their spherical functions are coincide.
\begin{prop}
Denote by $\varphi_z$ the spherical function of
$T_{z,\frac{1}{2}}$. Then we have
\begin{equation}\label{sphericalfunctionTz}
\varphi_z|_{S(2n)}(g)=\left(Reg^n(g)F_z^{(n)},F^{(n)}_z\right)_{L^{2}(X(n),\mu_1^{(n)})}.
\end{equation}
\end{prop}
\begin{proof}
Let $f_0\equiv 1$ be a unit vector, and let us consider $f_0$ as
an element of $L^{2}(X(n),\mu_t^{(n)})$. Then we find
$$
\left(T_{z,\frac{1}{2}}(g)f_0\right)(\bx)=z^{c(\bx;g)},\; \bx\in
X(n),\; g\in S(2n).
$$
If $g\in H(n)$, then $c(\bx;g)=0$. In this case we obtain that
$f_0$ is invariant under the action of $H(n)$, so $f_0$ can be
understood as the cyclic vector of the $S(2n)$-representation
$T_{z,\frac{1}{2}}|_{L^{2}(X(n),\mu_t^{(n)})}$. On the other hand,
the $S(2n)$-representation
$T_{z,\frac{1}{2}}|_{L^{2}(X(n),\mu_t^{(n)})}$ is equivalent to
$Reg^n$. This representation, $Reg^n$, acts in the space
$L^{2}(X(n),\mu_1^{(n)})$, and from the proof of Proposition
\ref{Proposition3.3-3.2} we conclude that the cyclic vector of the
$S(2n)$-representation $Reg^n$ is $F_z^{(n)}$ defined by formula
(\ref{3.3-3.2.1}). This gives expression for the spherical
function of $T_{z,\frac{1}{2}}$ in the statement of the
Proposition.
\end{proof}
\section{Definition of $z$-measures associated with the representations $T_{z,\frac{1}{2}}$}

\subsection{The space $C(S(2n), H(n))$}
Consider the set of functions on $S(2n)$ constant on each double
coset $H(n)gH(n)$ in $S(2n)$. We shall denote this set by
$C(S(2n),H(n))$. Therefore,
$$
C(S(2n),H(n))=\left\{f|f(hgh')=f(g),\;\mbox{where}\; h, h'\in
H(n),\;\; \mbox{and}\;\; g\in S(2n)\right\}.
$$
We equip $C(S(2n),H(n))$ with the scalar product $<.,>_{S(2n)}$
defined by
$$
<f_1,f_2>_{S(2n)}=\frac{1}{|S(2n)|}\sum\limits_{g\in
S(2n)}f_1(g)\overline{f_2(g)}.
$$
\begin{prop}\label{Proposition4.1-4.1.1}
The space $C(S(2n), H(n))$ is isometrically isomorphic to the
space $L^2(X(n),\mu_1^{(n)})^{H(n)}$ defined as a subset of
functions from $L^2(X(n),\mu_1^{(n)})$ invariant with respect to
the right action of $H(n)$,
\begin{equation}
\begin{split}
L^2(X(n),\mu_1^{(n)})^{H(n)}=\biggl\{f|f\in L^2(X(n),\mu_1^{(n)}),
f(\bx)=f(\bx\cdot h),\\
\;\mbox{where}\; \bx\in X(n),\;\;
\mbox{and}\;\; h\in H(n)\biggr\}.
\end{split}
\nonumber
\end{equation}
\begin{proof} The claim of the Proposition is almost trivial. Indeed,
the fact that $C(S(2n),H(n))$ is isomorphic to
$L^2(X(n),\mu_1^{(n)})^{H(n)}$ is obvious from the definition of
these spaces. We have
\begin{equation}
\begin{split}
<f_1,f_2>_{S(2n)}&=\frac{1}{|S(2n)|}\sum\limits_{g\in
S(2n)}f_1(g)\overline{f_2(g)}\\
&=\frac{1}{|X(n)|}\sum\limits_{\bx\in
X(n)}f_1(\bx)\overline{f_2(\bx)}=(f_1,f_2)_{L^2(X(n),\mu_1^{(n)})^{H(n)}},
\end{split}
\nonumber
\end{equation}
for any two functions $f_1, f_2$ from $C(S(2n),H(n))$. Therefore,
the isomorphism between  $C(S(2n),H(n))$ and
$L^2(X(n),\mu_1^{(n)})^{H(n)}$ is isometric.
\end{proof}
\end{prop}
\subsection{The  spherical functions of the Gelfand pair
$(S(2n),H(n))$} It is known (see Macdonald \cite{macdonald},
Section VII.2) that $(S(2n),H(n))$ is a Gelfand pair. In
particular, this implies that there is an orthogonal basis
$\{w^{\lambda}\}$ in $C(S(2n),H(n))$ whose elements,
$w^{\lambda}$, are the spherical functions of $(S(2n),H(n))$. The
elements $w^{\lambda}$ are parameterized by Young diagrams with
$n$ boxes, and are defined by
$$
w^{\lambda}(g)=\frac{1}{|H(n)|}\sum\limits_{h\in
H(n)}\chi^{2\lambda}(gh),
$$
see Macdonald \cite{macdonald}, Sections VII.1 and VII.2. Here
$\chi^{2\lambda}$ is the character of the irreducible
$S(2n)$-module corresponding to
$2\lambda=(2\lambda_1,2\lambda_2,\ldots )$. By Proposition
\ref{Proposition4.1-4.1.1} the  spherical functions $w^{\lambda}$
define an orthogonal basis in $L^2(X(n),\mu_1^{(n)})^{H(n)}$.
Besides, the zonal spherical functions $w^{\lambda}$ satisfy to
the following relations
\begin{equation}
w^{\lambda}(e)=1,\;\;\mbox{for any}\;\;\lambda\in\Y_n,
\end{equation}
\begin{equation}\label{4.2-4.1}
(w^{\lambda},w^{\mu})_{L^2(X(n),\mu_1^{(n)})^{H(n)}}=\frac{\delta_{\lambda,\mu}}{\dim
2\lambda},
\end{equation}
\begin{equation}\label{4.2-4.2}
\frac{1}{|X(n)|}\sum\limits_{\bx\in X(n)}w^{\lambda}(\bx\cdot
g)w^{\mu}(\bx)=\delta^{\lambda,\mu}\frac{w^{\lambda}(g)}{\dim
2\lambda},\;\;g\in S(2n).
\end{equation}
Here $\dim 2\lambda=\chi^{2\lambda}(e)$. The relations just
written above follow from general properties of  spherical
functions, see Macdonald \cite{macdonald}, Section VII.1.
\subsection{The z-measures $M^{(n)}_{z,\frac{1}{2}}$ of the representation $T_{z,\frac{1}{2}}$}

\begin{defn}\label{Definition of zmeasure} Let $z$ be a nonzero complex number, $\lambda$ be a
Young diagram with $n$ boxes, and let
$$
\tilde{w}^{\lambda}=\left(\dim 2\lambda\right)^{1/2}\cdot
w^{\lambda}
$$
 be the normalized zonal spherical
function of the Gelfand pair $(S(2n), H(n))$ parameterized by
$\lambda$. Set
\begin{equation}\label{4.3-4.3.1}
M_{z,\frac{1}{2}}^{(n)}(\lambda)=\left|(F_z^{(n)},\tilde{w}^{\lambda})_{L^{2}(X(n),\mu_1^{(n)})}\right|^2,
\end{equation}
where $F_z^{(n)}$ is a vector from $L^2(X(n),\mu_1^{(n)})$ defined
by equation (\ref{3.3-3.2.1}). The function
$M_{z,\frac{1}{2}}^{(n)}$ defined on the set of Young diagrams
with $n$ boxes is called the $z$-measure of the representation
$T_{z,\frac{1}{2}}$.
\end{defn}
The relation with the representation $T_{z,\frac{1}{2}}$ is clear
from the following Proposition.
\begin{prop} Denote by $\varphi_z$ the spherical function of
$T_{z,\frac{1}{2}}$. We have
\begin{equation}\label{4.3-4.4}
\varphi_z|_{S(2n)}(g)=\sum\limits_{|\lambda|=n}M_{z,\frac{1}{2}}^{(n)}(\lambda)w^{\lambda}(g),\;\;
g\in S(2n).
\end{equation}
\end{prop}
\begin{proof} The functions $\{\tilde{w}^{\lambda}\}$ define an
orthonormal basis in $L^2(X(n),\mu_1^{(n)})^{H(n)}$. On the other
hand, we can check that $F_z^{(n)}$ is an element of
$L^2(X(n),\mu_1^{(n)})^{H(n)}$. Therefore, we must have
\begin{equation}\label{4.3-4.5}
F_z^{(n)}(\bx)=\sum\limits_{|\lambda|=n}a_z^{(n)}(\lambda)\tilde{w}^{\lambda}(\bx),\;
\bx\in X(n).
\end{equation}
We insert expression (\ref{4.3-4.5}) into formula
(\ref{sphericalfunctionTz}). This gives
\begin{equation}
\varphi_z|_{S(2n)}(g)=\frac{1}{|X(n)|}\sum\limits_{\bx\in
X(n)}\sum\limits_{|\lambda|=n}\sum\limits_{|\mu|=n}\overline{a_z^{(n)}(\lambda)}a_z^{(n)}(\mu)\tilde{w}^{\lambda}(\bx\cdot
g)\tilde{w}^{\mu}(\bx). \nonumber
\end{equation}
Using equation (\ref{4.2-4.2}) we find that
\begin{equation}\label{M1}
\varphi_z|_{S(2n)}(g)=\sum\limits_{|\lambda|=n}|a_z^{(n)}(\lambda)|^2w^{\lambda}(
g).
\end{equation}
From equations (\ref{4.3-4.3.1}) and (\ref{4.3-4.5}) we see that
\begin{equation}\label{M2}
M_{z,\frac{1}{2}}^{(n)}(\lambda)=|a_z^{(n)}(\lambda)|^2,
\end{equation} which gives the formula in the statement of the
Proposition.
\end{proof}
\begin{cor} We have
$$
\sum\limits_{|\lambda|=n}M_{z,\frac{1}{2}}^{(n)}(\lambda)=1,
$$
i.e. $M_{z,\frac{1}{2}}^{(n)}(\lambda)$ can be understood as a
probability measure on the set of Young diagrams with $n$ boxes.
\end{cor}
\begin{proof}
This follows from equations (\ref{M1}), (\ref{M2}), and from the
fact that
$$
\varphi_z|_{S(2n)}(e)=w^{\lambda}(e)=1.
$$
\end{proof}
\subsection{An explicit formula for $M^{(n)}_{z,\frac{1}{2}}$}

\begin{prop}\label{TheoremExplicitFormulaForMz}(Olshanski \cite{olshanskiletter}) The $z$-measure
$M_{z,\frac{1}{2}}^{(n)}(\lambda)$ admits the following explicit
formula
$$
M_{z,\frac{1}{2}}^{(n)}(\lambda)=\frac{n!}{\left(\frac{z\bar
z}{2}\right)_n}\cdot
\frac{\prod\limits_{(i,j)\in\lambda}(z+2(j-1)-(i-1))(\bar{z}+2(j-1)-(i-1))}{h(2\lambda)},
$$
where $h(2\lambda)$ denotes the product of the hook-lengths of
$2\lambda=(2\lambda_1,2\lambda_2,\ldots )$, and $(.)_n$ stands for
the Pochhammer symbol,
$$
(a)_n=a(a+1)\ldots (a+n-1)=\frac{\Gamma(a+n)}{\Gamma(a)}.
$$
In particular, it follows that $M_{z,\frac{1}{2}}^{(n)}(\lambda)$
is exactly the $z$-measure with the Jack parameter $\theta=1/2$ in
the notation of Section \ref{Sectionztheta},
$$
M_{z,\frac{1}{2}}^{(n)}(\lambda)=M^{(n)}_{z,\bar
z,\theta=\frac{1}{2}}(\lambda).
$$
\end{prop}
\begin{proof}
We start from formula (\ref{4.3-4.3.1}), and observe that this
formula can be rewritten as
\begin{equation}\label{Mzt}
M_{z,\frac{1}{2}}^{(n)}(\lambda)=\frac{1}{[(2n)!]^2}\left|\widehat{(F_z^{(n)},\tilde{w}^{\lambda})}\right|^2,
\end{equation}
where $F_z^{(n)}$, $\tilde{w}^{\lambda}$ are understood as two
functions from $C(S(2n),H(n))$, and the scalar product
$\widehat{(f_1,f_2)}$ is defined by
$$
\widehat{(f_1,f_2)}=\sum\limits_{g\in
S(2n)}f_1(g)\overline{f_2(g)}.
$$
To compute the scalar product in equation (\ref{Mzt}), we use the
characteristic map,
$$
C(S(2n), H(n))\stackrel{ch''}{\longrightarrow} \Lambda_{\C}^n,
$$
 introduced in Macdonald \cite{macdonald},
Section VII.2. Her $\Lambda^n$ denotes the set of the homogeneous
symmetric polynomials of degree $n$, and $\Lambda^n_{\C}$ is the
linear span of these polynomials with complex coefficients. the
characteristic map, $ch''$, is defined by
\begin{equation}\label{ch''}
ch''(f)=|H(n)|\sum\limits_{|\rho|=n}z_{\rho}^{-1}2^{-l(\rho)}p_{\rho}f(\rho).
\end{equation}
Here the symbol $z_{\rho}$ is defined by
$$
z_{\rho}=\prod\limits_{i\geq 1}i^{m_i}\cdot m_i!,
$$
where $m_i=m_i(\rho)$ is the number of parts of $\rho$ equal to
$i$. In equation (\ref{ch''}) $l(\rho)$ stands for the number of
nonzero parts in $\rho$, $p_{\rho}=p_{\rho_1}p_{\rho_2}\ldots $,
where $p_k$ stands for $k$th power sum, and $f(\rho)$ is the value
of $f$ at elements of the double coset parameterized by the Young
diagram $\rho$, see Macdonald, Section VII.2. The map $ch''$ is an
isometry of $C(S(2n), H(n))$ onto $\Lambda_{\C}^n$. Therefore,
\begin{equation}\label{t1-1}
M_{z,\frac{1}{2}}^{(n)}(\lambda)=\frac{1}{[(2n)!]^2}\left|(ch''(F_z^{(n)}),ch''(\tilde{w}^{\lambda}))\right|^2,
\end{equation}
where the scalar product is defined by
$$
(p_{\rho},p_{\sigma})=\delta_{\rho\sigma}2^{l(\rho)}z_{\rho}.
$$
It remains to find $ch''(F_z^{(n)})$, $ch''(\tilde{w}^{\lambda})$,
and to compute the scalar product in the righthand side of
equation (\ref{t1-1}). We have
\begin{equation}\label{t1-2}
ch''(\tilde{w}^{\lambda})=(\dim
2\lambda)^{1/2}J_{\lambda}^{(\alpha=2)},
\end{equation}
where $J_{\lambda}^{(\alpha)}$ stands for the Jack polynomial with
the Jack parameter $\alpha$ parameterized by the Young diagram
$\lambda$ (in notation of Macdonald, Section VI). In order to find
$ch''(F_z^{(n)})$ it is enough to obtain a formula for
$ch''(N^{[.]_n})$. We have
$$
ch''(N^{[.]_n})=|H(n)|\sum\limits_{|\rho|=n}z_{\rho}^{-1}p_{\rho}\left(\frac{N}{2}\right)^{l(\rho)}.
$$
Since
$$
\left(\frac{N}{2}\right)^{l(\rho)}=p_{\rho}(\underset{N/2}{\underbrace{1,\ldots,1}}),
$$
we can use equation (1.4) in Section I.4 of Macdonald
\cite{macdonald}, and  write
$$
\left(\frac{N}{2}\right)^{l(\rho)}=|H(n)|\left\{\prod\limits_{i=1}^{\infty}(1-x_i)^{-\frac{N}{2}}\right\}_n.
$$
Here $\{.\}_n$ denotes the component of degree $n$. Now we have
$$
\prod\limits_{i=1}^{\infty}(1-x_i)^{-\frac{N}{2}}=\sum\limits_{\lambda}
\frac{1}{h(2\lambda)}J_{\lambda}^{(2)}(x)J_{\lambda}^{(2)}(\underset{N/2}{\underbrace{1,\ldots,1}}).
$$
The value
$J_{\lambda}^{(2)}(\underset{N/2}{\underbrace{1,\ldots,1}})$ is
known,
$$
J_{\lambda}^{(2)}(\underset{N/2}{\underbrace{1,\ldots,1}})=\prod\limits_{(i,j)\in\lambda}(N+2(j-1)-(i-1)).
$$
This gives us the following formula
$$
\prod\limits_{i=1}^{\infty}(1-x_i)^{-\frac{N}{2}}=
\sum\limits_{|\lambda|=n}
\frac{1}{h(2\lambda)}J_{\lambda}^{(2)}(x)\prod\limits_{(i,j)\in\lambda}(N+2(j-1)-(i-1)),
$$
and we obtain
\begin{equation}\label{t1-3}
\begin{split}
&ch''(F_z^{(n)})=\left(\frac{1\cdot 3\cdot\ldots\cdot
(2n-1)}{t(t+2)\ldots (t+2n-2)}\right)^{1/2}ch''(z^{[.]_n})\\
&=\left(\frac{1\cdot 3\cdot\ldots\cdot (2n-1)}{t(t+2)\ldots
(t+2n-2)}\right)^{1/2}|H(n)|\sum\limits_{|\lambda|=n}
\frac{1}{h(2\lambda)}J_{\lambda}^{(2)}(x)\prod\limits_{(i,j)\in\lambda}(N+2(j-1)-(i-1)).
\end{split}
\end{equation}
Finally, using the orthogonality relation
$$
(J_{\lambda}^{(2)},J_{\mu}^{(2)})=\delta_{\lambda\mu}h(2\lambda)
$$
we find from equations (\ref{t1-1})-(\ref{t1-3}) that
\begin{equation}
\begin{split}
M_{z,\frac{1}{2}}^{(n)}(\lambda)=&\frac{|H(n)|^2}{[(2n)!]^2}\left(\frac{1\cdot
3\cdot\ldots\cdot
(2n-1)}{t(t+2)\ldots (t+2n-2)}\right)\dim 2\lambda\\
&\times\prod\limits_{(i,j)\in\lambda}(z+2(j-1)-(i-1))(\bar
z+2(j-1)-(i-1)).
\end{split}
\nonumber
\end{equation}
Noting that $|H(n)|=2^nn!$, and that $\dim
2\lambda=\frac{(2n)!}{h(2\lambda)}$ we arrive to the first formula
in the statement of the Proposition. The fact that
$M_{z,\frac{1}{2}}^{(n)}(\lambda)$ coincides with the $z$-measure
with $\theta=1/2$ in the notation of the Section
\ref{Sectionztheta} can now be checked directly using formulae for
$H(\lambda,\theta)$ and $H'(\lambda,\theta)$ stated in Section
\ref{Sectionztheta}.

\end{proof}

\end{document}